\documentclass{amsart}
\usepackage{amssymb}
\usepackage{graphicx}
\newtheorem{theorem}{Theorem}
\newtheorem{definition}{Definition}
\newtheorem{proposition}{Proposition}
\newtheorem{lemma}{Lemma}

\newtheorem{rmk}{Remark}
\newenvironment{remark}{\begin{rmk}\rm}{\end{rmk}}

\renewcommand{\ne}{\not =}

\title{Infinitesimal Poincar\'{e}-Bendixson problem in dimension three}
\author{C. Alonso-Gonz\'{a}lez}
\author{F. Cano}
\author{R. Rosas}
\date{26 November 2012 }
\begin{document}
\maketitle
\section*{Abstract} We describe the sets of accumulation of secants for orbits of real analytic
vector fields in dimension three having the origin as only $\omega$-limit point.
It is a kind of infinitesimal Poincaré-Bendixson problem in dimension three.
These sets have structure of cyclic graph when the singularities are isolated under
one blow-up. In the case of hyperbolic reduction of singularities with conditions
of Morse-Smale type, we prove that the accumulation set is at most a single poly-cycle
isomorphic to ${\mathbb S}^1$. 

\section{Introduction}
Let
$$
\xi=a(x,y,z)\frac{\partial}{\partial x}+
b(x,y,z)\frac{\partial}{\partial y}+
c(x,y,z)\frac{\partial}{\partial z}
$$
be a real analytic vector field defined in a neighborhood of the origin of ${\mathbb R}^3$ and assume that the origin is an equilibrium point of $\xi$. Consider an orbit $\gamma$ of $\xi$ such that the origin is the only $\omega$-limit point of $\gamma$, that is $\lim_{t\rightarrow +\infty}\gamma(t)=(0,0,0)$.
The {\em secants\/} of $\gamma$ are the vectors $\gamma(t)/\vert\vert\gamma(t)\vert\vert\in {\mathbb S}^2$.
We say that $\gamma$ has {\em tangent} at the origin if the limit of secants $$\lim_{t\mapsto +\infty}\gamma(t)/\vert\vert\gamma(t)\vert\vert$$ exists. In case we have no tangent, it is a natural question to ask for a description of the set of accumulation of secants
$$
\mbox{Sec}(\gamma)=\bigcap_{s}\overline{
\{\gamma(t)/\vert\vert\gamma(t)\vert\vert;\, t\geq s\}
}\subset {\mathbb S}^2.
$$
This problem can be considered as an infinitesimal version of the classical Poincar\'{e}-Bendixson Theorem. The question raised from the study in \cite{Can-M-S} of the oscillation properties of trajectories of analytic vector fields. In the particular case of gradient vector fields, it is known \cite{Kur-M-P} that all the orbits have a well defined tangent, that is, the set $\mbox{Sec}(\gamma)$ is a single point, but for general vector fields, the problem is widely open.

In this paper we give a result for generic absolutely isolated singularities. In this case ``up to further blow-ups" the accumulation set is at most a single cycle.
\begin{theorem}
 \label{th:absi}
 Assume that $\xi$ has a generic absolutely isolated singularity at the origin. Then, up to perform finitely many blow-ups, the set of accumulation of secants $\mbox{Sec}(\gamma)$ is of one of the following types
\begin{enumerate}
\item A single point (case of iterated tangents).
\item A periodic orbit.
\item A poly-cycle homeomorphic to ${\mathbb S}^1$.
\end{enumerate}
\end{theorem}
Absolutely isolated singularities of vector fields have been introduced in \cite{Cam-C-S} for the complex case. The definition is the same one for the real case: the singularity is isolated and, after any finite sequence of blow-ups centered at singular points, we obtain isolated singular points and the exceptional divisor is invariant. In \cite{Cam-C-S} it is shown that we get a reduction of singularities of the vector field after finitely many blow-ups (the result is also true for the real case). We obtain in this way a blown-up ambient space which is a manifold with corners jointly with a line foliation on it. The generic conditions we assume are explained in Sections  \ref{sec:ms} and \ref{sec:weights} and they are a three-dimensional version of Morse-Smale conditions. Now, Theorem \ref{th:absi} is a consequence of
the main result in this paper, that we state as follows
\begin{theorem}
 \label{th:msth}
 Consider an oriented line foliation $\mathcal L$ on a three-dimensional real analytic manifold $M$ with corners and spherical frontier $\partial M$. Assume that $\mathcal L$ is of weak Morse-Smale type. For any
 parameterized orbit $\gamma$ contained in $ M\setminus\partial M$  and whose  $\omega$-limit set $\omega(\gamma)$ is contained in $\partial M$, we have that $\omega(\gamma)$ is either a single point or it is homeomorphic to ${\mathbb S}^1$.
\end{theorem}

In the body of the paper we explain the terms and give a proof of this result. The general technique we introduce is the construction of the so-called Poincar\'{e}-Bendixson traps, that work as in the classical proof of Poincar\'{e}-Bendixson Theorem.

It is not known how can be the $\omega$-limit set $\omega(\gamma)$ without the stated conditions, even if some constructions we have seen with S. Ib\'{a}\~{n}ez suggest that it should be possible to construct examples where $\omega(\gamma)$ is essentially different than the accumulation sets described in the plane Poincar\'{e}-Bendixson theorem. Anyway, it seems that some constructions of Belotto \cite{Bel} for limits of periodic orbits in families could be adapted to generate examples of accumulation sets $\mbox{Sec}(\gamma)$ in any real analytic subset of ${\mathbb S}^2$. In \cite{Pan-Rou} we find other related results for higher dimensional versions of Poincar\'{e}-Bendixson theorem.

\section{Line foliations in manifolds with corners}
Since we are going to perform several blow-ups of ${\mathbb R}^3$, it will be useful to consider real analytic manifolds with corners in dimension three and oriented line foliations on them. Let us precise the definition we take here, to prevent possible confusion with other definitions in the literature.
A {\em real analytic manifold of dimension $n$ with corners } is a Hausdorff topological space $M$ jointly with a stratification by a flag of closed sets
$$
M\supset \partial M\supset \partial^2 M\supset\cdots\supset \partial^nM
$$
and an atlas ${\mathcal A}$ as follows. A chart $(U,\phi)$ in $\mathcal A$ is a homeomorphism $\phi:U\rightarrow V$ between an open set $U\subset M$ and an open set $V\subset {\mathbb R}_{\geq 0}^n$ such that
$$
\phi(U\cap\partial^kM)=V\cap \bigcup_{i_1<i_2<\cdots<i_k}L_{i_1i_2\cdots i_k}; \quad L_{i_1i_2\cdots i_k}=\bigcap_{\ell=1}^k(x_{i_\ell}=0).
$$
Moreover, each connected component $S^{k}_j$ of $\partial^kM\setminus \partial^{k+1}M$ satisfies that $\phi(U\cap S^{k}_j)$ is contained in only one of the linear subspaces $L_{i_1i_2\cdots i_k}$.
The compatibility condition between two charts $(U,\phi)$ and $(U',\psi)$ is expressed by saying that
$$
\psi\circ\phi^{-1}:\phi(U\cap U')\rightarrow \psi(U\cap U')
$$
extends to an analytic map between open sets in ${\mathbb R}^n$. The adherence $D^k_j$ of each $S^{k}_j$ inherits a structure of real analytic manifold with corners. We also say that each $D^k_j$ is a {\em component } of $\partial^kM$. Finally we say that $\partial M$ is the {\em frontier} of $M$.

In this paper we consider mainly three-dimensional analytic manifolds $M$ with corners and {\em spherical frontier}, that is such that $\partial M$ is homeomorphic to ${\mathbb S}^2$. In this case we say that $\partial M$ is the {\em exceptional divisor of $M$}, the points of $\partial M\setminus\partial^2M$ are the {\em plane frontier points}.  We call $\partial^2 M$ the {\em skeleton} of $M$, where the points of $\partial^2M\setminus\partial^3M$ are the {\em angle points} and
the points of $\partial^3M$ are the {\em corner points of $M$}.

We get an important family of examples of three-dimensional analytic manifolds $M$ with corners and spherical frontier
as follows. Let $M_1$ be the blow-up of ${\mathbb R}^3$ in polar coordinates at the origin
$$
{\mathbb R}^3\stackrel{\sigma}{\leftarrow} M_1={\mathbb R}_{\geq 0}\times {\mathbb S}^2
$$
where $\sigma(t,{\mathbf p})=t{\mathbf p}$. We obtain now $M$ by performing successive blow-ups with centers given by points in the exceptional divisor
$$
{\mathbb R}^3\stackrel{\sigma_1}{\leftarrow} M_1 \stackrel{\sigma_2}{\leftarrow} M_2\stackrel{\sigma_3}{\leftarrow}\cdots
\stackrel{\sigma_k}{\leftarrow}M_k=M.
$$

We define  an {\em oriented line foliation with singularities} $\mathcal L$ on $M$ to be a {\em maximal foliated atlas} ${\mathcal A}_{\mathcal L}$. A foliated chart has the form $(U,\phi,\xi)$, where $(U,\phi)$ is a chart for $M$ and $\xi$ is an analytic vector field in an open neighborhood of $\phi(U)$ such that each of the $L_{i_1i_2\cdots i_k}$ are invariant for $\xi$.   The compatibility condition for $(U,\phi,\xi)$ and $(U',\psi,\xi')$ implies that $\xi$ and $h\xi'$
are $\psi\circ\phi^{-1}$-related, where $h$ is a positive real analytic function in a neighborhood of the common domain of definition. Given $p\in M$, we denote ${\mathcal L}(p)\subset T_pM$ the tangent subspace spanned by $\xi(p)$. If $p$ is an equilibrium point we have ${\mathcal L}(p)=\{0\}$, otherwise it is a tangent line.
Note that $\mathcal L$ induces an oriented line foliation ${\mathcal L}^k_j$ in each component $D^{k}_j$ of $\partial^kM$. Finally, a (positive) {\em parameterized orbit} of $\mathcal L$ is a real analytic application
$$
\gamma:(-\infty,+\infty)\rightarrow M
$$
such that $\gamma'(t)$ is a non-null tangent vector positively proportional to  $\xi(\gamma(t))$ and, for any $\gamma(t)\in U$, the support $\vert\gamma\vert$ (the image of $\gamma$) contains the maximal integral curve of $\xi$ through $\gamma(t)$ in $U$.  Let us note that $\vert \gamma\vert$ is contained in each $D^k_j$ such that $\vert \gamma\vert\cap D^k_j\ne \emptyset$.
We say that $\gamma$ is a {\em periodic orbit} if there are $a<b$ such that $\gamma(a)=\gamma(b)$; in this case $\vert\gamma\vert=\gamma([a,b])$ and it is homeomorphic to ${\mathbb S}^1$.

Let us recall that the $\omega$-limit set $\omega(\gamma)$ of a parameterized orbit $\gamma$ is defined by
$$
\omega(\gamma)=\bigcap_t \overline{\gamma(t,+\infty)}.
$$
In the case of a periodic orbit it is obvious that $\omega(\gamma)=\vert \gamma \vert$. In  a general situation, the $\omega$-limit set may be quite complicated.

{\bf Assumption.} From now on, we take an oriented line foliation $\mathcal L$ on $M$, where $M$ is a three-dimensional real analytic manifold with spherical frontier. We consider a parameterized orbit $\gamma$ such that $\vert\gamma\vert\subset M\setminus\partial M$  and $\omega(\gamma)\subset \partial M$.
\begin{remark}It is known that $\omega(\gamma)$ is a compact and connected set which is invariant by $\mathcal L$, in the sense that it is a union of equilibrium points and orbits.\end{remark}

\section{Poincar\'{e}-Bendixson arguments}
 In this section, we adapt  the classical arguments in the proof of Poincar\'{e}-Bendixson Theorem to the three-dimensional case.

Let $[0,1]\subset {\mathbb R}$ be the closed unit interval. A {\em Poincar\'{e}-Bendixson wall} for ${\mathcal L}$ is a continuous map
$
w:[0,1]\times[0,1]\rightarrow M
$
with $w^{-1}(\partial M)=[0,1]\times\{0\}$,
such that the restriction $$w\vert_{(0,1)\times(0,1)}:(0,1)\times(0,1)\rightarrow M$$ defines a differentiable embedding into $M\setminus \partial M$ tangent to $\mathcal L$ and moreover $w$ satisfies one of the next conditions:
\begin{enumerate}
\item[I.]{\em Closed wall:} $w$ factorizes through a topological embedding $$\tilde w: {\mathbb S}^1\times [0,1]\rightarrow M.$$
\item[II.]{\em Semi-open wall:} $w$ factorizes through a topological embedding $$\tilde w: [0,1]\times [0,1]/\sim\rightarrow M,$$
    where $\sim$ identifies $(0,0)$ and $(1,0)$.
\item[III.]{\em Open wall:} $w$ is a topological embedding $w:[0,1]\times[0,1]\rightarrow M$.
\end{enumerate}
The {\em base of the wall} is the restriction $\beta$ of $w$ to $[0,1]\times\{0\}$, it is always an invariant path for $\mathcal L$ in $\partial M$.
In the cases of types I and II we obtain a simple loop $\tilde \beta:{\mathbb S}^1\rightarrow \partial M$ given by $\tilde\beta(\exp {( 2\pi i t)})=\beta(t,0)$.

A {\em transversal door} for $\mathcal L$ is a continuous map
$
\delta:[0,1]\times[0,1]\rightarrow M
$
with $\delta^{-1}(\partial M)=[0,1]\times\{0\}$,
such that  $\delta\vert_{[0,1]\times(0,1]}:[0,1]\times(0,1]\rightarrow M$ defines a differentiable embedding into $M\setminus \partial M$ transversal to $\mathcal L$ and moreover $\delta$ satisfies one of the next conditions:
\begin{enumerate}
\item[T.]{\em Triangular door:} $\delta$ factorizes through a topological embedding
$$\tilde \delta: [0,1]\times [0,1]/\approx\rightarrow M,$$
    where $\approx$ identifies $(0,0)$ and $(t,0)$ for any $0\leq t\leq 1$.
\item[S.]{\em Square door:} $\delta$ is a topological embedding $\delta: [0,1]\times [0,1]\rightarrow M$.
\end{enumerate}
\begin{definition}
Let $w$ be a wall of type II, respectively III, and $\delta$ a door of type T, respectively S. We say that a continuous map $
\tau:[0,1]\times[0,1]\rightarrow M
$ is a {\em Poincar\'{e}-Bendixson trap of type II, respectively III, with wall $w$ and door $\delta$ }if
$$
\tau (t,s)=\left\{
\begin{array}{ccc}
w(2t,s)&\mbox{ if }& 0\leq t\leq 1/2\\
\delta(2t-1,s)&\mbox{ if }& 1/2\leq t\leq 1
\end{array}
\right.
$$
and it factorizes through a topological embedding $\tilde \tau:{\mathbb S}^1\times[0,1]\rightarrow M$ in such a way that
$
\tilde\tau (\exp(2\pi i t),s)=\tau (t,s)$. A {\em Poincar\'{e}-Bendixson trap $\tau$ of type I associated to a closed wall $w$} is by definition $\tau=w$. The {\em base $B_\tau\subset \partial M$} of $\tau$ is the image $B_\tau=\tilde\tau ({\mathbb S}^1\times\{0\})$. In cases of types II and III, we denote by $D_\tau$ the image of $\delta$ and we put $D_\tau=\emptyset$ if $\tau$ has type I.
\end{definition}

Let us remark that since $\partial M$ is homeomorphic to ${\mathbb S}^2$ and $B_\tau$ defines a Jordan curve in $\partial M$, then $\partial M\setminus B_\tau$ has exactly two connected components and their common frontier is $B_\tau$.

\begin{lemma}
\label{lema:uno}
Let $
\tau
$ be a Poincar\'{e}-Bendixon trap for $\mathcal L$, then $\omega(\gamma)$ is outside one of the connected components of  $\partial M\setminus B_\tau$.
\end{lemma}
\begin{proof} We can find a connected open set $U\subset M$ with $\partial M\subset U$ and such that $$U\setminus \tilde\tau( {\mathbb S}^1\times [0,1])$$ is the union of two connected components $U_1$ and $U_2$ satisfying the following properties
\begin{enumerate}
\item $\bar U_1\setminus U_1=\bar U_2\setminus U_2\subset \tilde\tau( {\mathbb S}^1\times [0,1])$, where
$\bar U_i$ is the topological closure of $U_i$ in $U$, for $i=1,2$.
\item $U_1\cap\partial M$ and $U_2\cap\partial M$  are the two connected components of $\partial M\setminus B_\tau$.
\end{enumerate}
Since $\omega(\gamma)\subset \partial M$, there is a time $t_0$ such that $\gamma(t)\in U$ for $t\geq t_0$. Assume that $\gamma(t_0)\in U_1$. If $\gamma(t)\in U_1$ for all $t\geq t_0$, we are done, since then $\omega(\gamma)$ is contained in
$$
\overline{U_1\cap\partial M}=( U_1\cap\partial M)\cup B_\tau
$$
and it has empty intersection with $U_2\cap \partial M$. If we exit $U_1$, it is not possible to return, thus we have two cases:
there is $t_1>t_0$ such that $\gamma(t_1)\in U_2$ or $\gamma(t)$ stands in $\tilde\tau({\mathbb S}^1\times [0,1])$ for $t\geq t_3>t_1$.
If there is $t_1>t_0$ such that $\gamma(t_1)\in U_2$, then there is a time $t_2$, $t_0<t_2<t_1$ such that $\gamma(t_2)\in \tau({\mathbb S}^1\times [0,1])$. Then necessarily $\gamma(t_2)$ belongs to the door $D_\tau$  and we have
$$
\gamma(t)\in U_1, \mbox{ for } t_0\leq t<t_2;\quad \gamma(t)\in U_2, \mbox{ for } t_2<t
$$
Note that we can ``cross'' the door only in one sense and thus, once we are in $U_2$, coming from $U_1$, it is not possible to exit $U_2$. We obtain that $\omega(\gamma)\subset \overline{U_2\cap\partial M}$
and it has empty intersection with  $U_1\cap \partial M$. Now, if $\gamma(t)$ stands in $\tilde\tau({\mathbb S}^1\times [0,1])$, we have that  $\omega(\gamma)\subset B_\tau$.
 \end{proof}

 Let us consider a component $D$ of $\partial M$. A {\em snail in $D$} is a topological embedding $
 {\zeta}:{\mathbb S}^1\rightarrow D
 $
 such that
  \begin{enumerate}
        \item The curve $\varsigma:(0,1)\rightarrow D$, defined by $\varsigma(t)={\zeta}(\exp \pi t)$, is the restriction of a parameterized orbit of $\mathcal L$ contained in $D$.
       \item The curve $d:(0,1)\rightarrow D$, defined by $d(t)={\zeta}(\exp \pi (t+1))$, is a differentiable curve transversal to $\mathcal L$.
  \end{enumerate}
  The image $\vert\zeta\vert$ of $\zeta$ is the {\em support} of the snail. We also denote $\zeta=(\varsigma,d)$.

 \begin{remark}  If ${\zeta}=(\varsigma,d)$ is a snail, the orbit $\varsigma$ has not the ``Rolle property'' introduced in Khovanskii's works \cite{Kho}.
 \end{remark}
 \begin{lemma}
 Let $\sigma$ be a periodic orbit with support $\vert\sigma\vert\subset \partial M$, there is a Poincar\'{e}-Bendixson trap $\tau$ of type I or II whose base is $\vert\sigma\vert$.
 Let ${\zeta}=(\varsigma,d)$ be a snail, there is a Poincar\'{e}-Bendixson trap $\tau$ of type III whose base is $\vert\zeta\vert$.
 \end{lemma}
 \begin{proof} In the case of a periodic orbit, the trap $\tau$ may be constructed as follows.  We consider a plane section $\Delta$  transversal to $\vert\sigma\vert$ in a point $P$ and an analytic curve $\Gamma\subset \Delta$ passing through $P$, note that $\Delta$ is fully transversal to $\mathcal L$ if it is small enough. Then we apply the Poincar\'{e} first return map to $\Gamma$.  The image $\Gamma'$ of $\Gamma$ is an analytic curve that locally coincides with $\Gamma$ (type I, closed trap) or intersects $\Gamma$ locally at the point $P$ (type II, semi-open trap). In the case of a snail, we start by taking  a plane section $\Delta$ transversal to $\mathcal L$ and containing $\vert d \vert$, in particular $P=\zeta (\exp 2\pi)$ and $Q=\zeta (\exp \pi)$ are in $\Delta$. We take as before a curve $\Gamma\subset \Delta$ passing through $P$ and we apply the first holonomy map along $\varsigma$ to obtain $\Gamma'\subset \Delta$, with $Q\in \Gamma'$, in this way we get a trap of type III, open trap.
 \end{proof}

  \begin{remark}
   \label{rk:entornodesnails}
   Let $\sigma$ be a periodic orbit with support $\vert\sigma\vert\subset \partial M$. Let $D$ be a component of $\partial M$ such that $\vert\sigma\vert\subset D$ and consider a small curve $\Gamma_D$ transversal to $\mathcal L$ in a point of $\vert\sigma\vert\subset D$. We can consider the Poincar\'{e} first  return map  on $\Gamma_D$; each point of $\Gamma_D$ gives either a periodic orbit or a snail. In this way we see that $\vert\sigma\vert$ has a fundamental system of neighborhoods $U_i\subset \partial M$, $i\in I$, where each $U_i$ is the union of $\vert\sigma\vert$ and a family $\{S_j\}_{j\in J}$ of periodic orbits and snail supports, different from $\vert\sigma\vert$. Moreover, if we take one of such $U_i$ and we denote by $W_j$ the adherence in $\partial M$ of the connected component of $\partial M\setminus S_j$ containing $\vert\sigma\vert$, we obtain that
   $$
   \vert\sigma\vert=\bigcap_{j\in J} W_j.
   $$
 \end{remark}
\begin{lemma} Let ${\zeta}=(\varsigma,d)$ be a snail. Then $\omega(\gamma)\cap \vert\zeta\vert=\emptyset$.
\end{lemma}
\begin{proof} Take a Poincar\'{e}-Bendixson trap $\tau$ whose base is $B_\tau=\vert\zeta\vert$. By Lemma \ref{lema:uno}, we know that $\omega(\gamma)$ is outside one of the connected components of $\partial M\setminus B_{\tau}$. But this is only possible if $\omega(\gamma)\cap \vert\zeta\vert=\emptyset$. Otherwise, take a point $p\in \omega(\gamma)\cap \vert\zeta\vert$. Assume first that $p={\zeta}(\exp 2\pi t)$, for $0\leq t\leq 1/2$ and hence $p=\varsigma(2t)$.
Then $\vert\varsigma\vert$ is contained in $\omega(\gamma)$. Taking $0<\epsilon <<1$, we have that $\varsigma(-\epsilon)$ and $\varsigma(1+\epsilon)$ are in  two distinct connected components of $\partial M\setminus B_{\tau}$, contradiction. If
$p={\zeta}(\exp 2\pi t)$, for $1/2<t<1$,  then $p=d(2t-1)$. Consider a parameterized orbit $\theta$ of $\mathcal L$ passing through $p$ and such that $\theta (0)=p$. In view of the transversal property of $d$ we have that $\theta(-\epsilon)$ and $\theta(+\epsilon)$ are in two distinct connected components of  $\partial M\setminus B_{\tau}$, getting a contradiction as above.
\end{proof}
\begin{proposition}
\label{prop:periodicorbit}
If there is a periodic orbit $\sigma$  contained in $\omega(\gamma)$, then $\omega(\gamma)=\vert\sigma\vert$.
\end{proposition}
\begin{proof} In view of Remark \ref{rk:entornodesnails}, let us consider a neighborhood  $U$ of $\vert\sigma\vert$ which is the union of $\vert\sigma\vert$ and a family $\{S_j\}_{j\in J}$ of periodic orbits and snail supports different from $\vert\sigma\vert$. By Lemma  \ref{lema:uno}, for each $j\in J$ there is a connected component $V_j$ of $\partial M\setminus S_j$ such that
$V_j\cap \omega(\gamma)=\emptyset$. Let us note that $V_j$ is the connected component of $\partial M\setminus S_j$ not containing $\vert\sigma\vert$.  Let us put $W_j=\partial M\setminus V_j$, by the preceding observation and Remark \ref{rk:entornodesnails} we have that
$$
\bigcap_{j\in J}W_j=\vert\sigma\vert.
$$
Now $\vert\sigma\vert\subset \omega(\gamma)\subset \bigcap_{j\in J}W_j$. We obtain that
 $\omega(\gamma)=\vert\sigma\vert$.
\end{proof}
\section{Graph structure of the omega-limit set}
Let us consider a parameterized orbit $\sigma$ of $\mathcal L$ contained in $\partial M$. We say that the oriented support $\vert \sigma \vert_+$ of $\sigma$ is an {\em edge} for $\mathcal L$ if the $\alpha$-limit set $\alpha(\sigma)$ and the $\omega$-limit set $\omega(\sigma)$ are both equilibrium points of $\mathcal L$. In this way we obtain an oriented graph $\mathcal G$ such that
 \begin{enumerate}
 \item The vertices of $\mathcal G$ are the equilibrium points of $\mathcal L$ in $\partial M$.
 \item The (oriented) edges are defined as above.
 \item An edge   $\vert \sigma \vert_+$ joins $p$ with $q$ if $p=\alpha(\sigma)$ and $q=\omega(\sigma)$.
 \end{enumerate}
 An edge $\vert \sigma \vert_+$ is called a {\em loop} if $\alpha(\sigma)=\omega(\sigma)$.   An edge $\sigma$ is called a {\em skeleton edge} if it is contained in $\partial^2 M$, otherwise it is a {\em trace edge}. Note that we have only finitely many skeleton edges.
   An {\em oriented segment of length $n$} of ${\mathcal C}(\gamma)$ is a finite sequence
 $$
 p_0,\sigma_1,p_1,\sigma_2,p_2,\ldots,p_{n-1},\sigma_n,p_n
 $$
 such that each $\sigma_i$ is an edge in ${\mathcal C}(\gamma)$ starting at $p_{i-1}$ and ending at $p_i$. The oriented segment is a {\em cycle} if $p_0=p_n$. A subgraph $\mathcal C$ of $\mathcal G$ is called {\em cyclic} if each edge of $\mathcal C$ is contained in a cycle in $\mathcal C$ (note that a loop defines a cycle).
 \begin{quote}{\bf Notation:} If there is no confusion, we denote just $\sigma$ the edge $\vert\sigma\vert_+$.
 \end{quote}
\begin{lemma}
 \label{lema:arista}
 Let $\sigma$ be a non periodic parameterized orbit of $\mathcal L$ whose support is contained in $\omega(\gamma)$. Then $\sigma$ is an edge for $\mathcal L$.
\end{lemma}
\begin{proof} Note that the support of $\sigma$ is contained in $\partial M$. If the omega limit set of $\sigma$ is not an equilibrium point, by the two-dimensional Poincar\'{e}-Bendixson Theorem, it is a poly-cycle with a return map that allows us to produce a snail from $\sigma$, contradiction. In the same way we see that the alpha limit set of $\sigma$ is also an equilibrium point. Then $\sigma$ is an edge.
\end{proof}
\begin{remark}Let us note that using the arguments of the proof of Lemma \ref{lema:arista}, we also have that $\sigma$ passes ``only once'' through each flow-box for $\mathcal L$.
  \end{remark}
  As a consequence of Lemma \ref{lema:arista} and in view of Proposition \ref{prop:periodicorbit}, we see that the omega limit set $\omega(\gamma)$ is the support of a subgraph ${\mathcal C}(\gamma)$ of $\mathcal G$ unless it coincides with the support of a periodic orbit.

  We devote the rest of this section to give a proof of the following statement:
  \begin{proposition}
  \label{prop:ciclico}
   Assume that $\mathcal L$ has only isolated equilibrium points and that $\omega(\gamma)$ does not coincide with the support of a periodic orbit. Then $\omega(\gamma)$ is the support of a connected cyclic subgraph ${\mathcal C}(\gamma)$ of $\mathcal G$.
 \end{proposition}

 We already know that ${\mathcal C}(\gamma)$ is connected. Hence, we only need to verify that ${\mathcal C}(\gamma)$ is a cyclic subgraph of $\mathcal G$. We will use the following characterization of cyclic graphs
 \begin{lemma}
  \label{lema:cyc}
  Let $\mathcal C$ be a connected oriented graph with a set $S$ of vertices. Then $\mathcal C$ is cyclic if and only if for any non-trivial partition $S=T\cup T'$ there is an edge starting at $T$ and ending at $T'$.
 \end{lemma}
 \begin{proof} Assume that $\mathcal C$ is cyclic. By connectedness, there is an edge $\sigma$ connecting a point $p\in T$ and another point $p'\in T'$. If $\sigma$ starts at $p$ and ends at $p'$, we are done. If $p'$ is the starting point and $p$ the end point of $\sigma$, we have a cycle
 $$
 p',\sigma,p=p_0,\sigma_1,p_1,\sigma_2,p_2,\ldots,p_{n-1},\sigma_n,p_n=p'.
 $$
 There is necessarily an index $j$ such that $p_j\in T$ and $p_{j+1}\in T'$ and then $\sigma_j$ has its origin in $T$ and its end point in $T'$.

 Conversely, take an edge $\sigma$ starting at a point $p$ and ending at $q$. Consider the set $T$
 of points $r\in S$ such that there is an oriented segment starting at $q$ and ending at $r$. If $p\in T$, we are done, since we can produce a cycle containing $\sigma$. Note that $q\in T$, since length zero segments are allowed. If $p\notin T$, then $T'=S\setminus T$ is nonempty and there is an edge $\tau$ starting at $t\in T$ and ending at $t'\in T'$, this gives a contradiction since $t'$ should also be in $T$.
   \end{proof}

   Let us start the proof of Proposition \ref{prop:ciclico}, by using the characterization given in Lemma \ref{lema:cyc}.
Denote by  $S\subset\mbox{Sing}{\mathcal L}$  the set of equilibrium points in $\omega(\gamma)$, that is $S$ is the set of vertices of the graph $\mathcal C$. It is a finite set contained in $\partial M$. Denote $S^*=\mbox{Sing}{\mathcal L}\setminus S$ and let us fix a non-trivial partition $S=T\cup T'$. Let $B_{S^*}$ be a compact neighborhood of $S^*$ such that $B_{S^*}\cap \omega(\gamma)=\emptyset$ and let us consider fundamental systems of compact neighborhoods $\{B_{T,n}\}_{n=0}^{\infty}$, respectively $\{B_{T',m}\}_{m=0}^{\infty}$ of $T$, respectively $T'$, such that for any $n\geq 0$ we have
$$
B_{T,n+1}\subset \stackrel{\circ}{B}_{T,n},\; B_{T',n+1}\subset \stackrel{\circ}{B}_{T',n}
$$
and moreover
$
B_{T,0}\cap B_{T',0}=B_{T,0}\cap B_{S^*}=B_{T',0}\cap B_{S^*}=\emptyset.
$

Since $S^*\cap \omega(\gamma)=\emptyset$, there is a time $t_0$ such that $\gamma(t)\notin B_{S^*}$ for any $t\geq t_0$. By an argument of compactness, we can assume that the parameterized orbit $\gamma(t)$ satisfies the following property:
\begin{quote} There is a fixed $\delta>0$ such that for any $t_0<t_1<t_2$ such that $\gamma(t_1)\in B_{T,0}$ and $\gamma(t_2)\in B_{T',0}$, then $t_2-t_1\geq \delta$.
\end{quote}
 Now, let us consider the set $A_n$ of closed intervals $I=[s_I,t_I]\subset {\mathbb R}_{\geq t_0}$ defined by the property that $t_0\leq s_I$ and
 $$
 \gamma(s_I)\in B_{T,n},\; \gamma(t_I)\in B_{T',n},\; \gamma((s_I,t_I)))\cap (B_{T,n}\cup B_{T',n})=\emptyset.
 $$
 Remark that $t_I-s_I\geq \delta$ and $I\cap I'=\emptyset$ for $I,I'\in A_n$, $I\ne I'$. Thus  $A_n$ is an ordinal equivalent to ${\mathbb N}$ with the ordering
 $$
 I\leq I'\Leftrightarrow s_I\leq s_{I'}.
 $$
 Let us note that if $m\geq n$ and  $J\in A_m$, there is at least one interval $I\in A_n$ such that $I\subset J$. Thus we have infinite sets $A_{n,m}$ for each $m\geq n\geq 0$ defined by
 $$
 A_{n,m}=\{I\in A_n;\mbox{ there is } J\in A_m \mbox{ such that } I\subset J\}.
 $$
  Note that $A_n=A_{n,n}\supset A_{n,m}\supset A_{n,m+1}$.
  Let us define the nonempty compact sets $K_{n,m}$ contained in $\omega(\gamma)\cap B_{T,n}$ for any pair $n\leq m$ as follows
 $$
 K_{n,m}=\overline{\{\gamma(s_{I}); I\in A_{n,m}\}}\cap {\partial M}.
 $$
 Note that $K_{n,m}\subset B_{T,n}\setminus B_{T,n+1}$ and $K_{n,m'}\subset K_{n,m}$ if $m'\geq m$. We have a nonempty compact set
 $$
 K_n=\bigcap_{m\geq n}K_{n,m}
 $$
 for each $n\geq 0$. Let us take a point $p\in K_0$. Denote by $\sigma_p$ a parameterized orbit of $\mathcal L$ such that $\sigma_p(0)=p$. Then $\sigma_p$ defines and edge of ${\mathcal C}$. Let us prove that the starting point of $\sigma_p$ is in $T$ and the end point of $\sigma_p$ is in $T'$. Denote by $\sigma_p^+$, respectively by $\sigma_p^-$, the restriction of $\sigma_p$ to $[0,+\infty)$, respectively to $(-\infty,0]$. Then, it is enough to prove that for any $n> 0$ we have
 $$
 \vert\sigma_p^-\vert\cap B_{T,n}\ne\emptyset,\quad \vert\sigma_p^+\vert\cap B_{T',n}\ne\emptyset.
 $$
 Let us fix $m>n>0$ and consider times $\alpha<0<\beta$ such that $$
 \sigma_p(\alpha)\in B_{T\cup T',m}, \; \sigma_p(\beta)\in B_{T\cup T',m}, \; \sigma_p((\alpha,\beta))\cap B_{T\cup T',m}=\emptyset,$$
 where $B_{T\cup T',m}=B_{T,m}\cup B_{T',m}$. Take a sequence of intervals $I_k\in A_{0,n}$ such that $I_k\subset J_k$ where $J_k\in A_n$ and $\lim_{k\rightarrow \infty}\gamma(s_{I_k})=p$. Now, we consider a flow-box $C$ for $\mathcal L$ around the central line $\sigma_p[\alpha,\beta]$. We take $C$ such that
 the initial and final parts of the flow-box are in the interior of $B_{T\cup T',n}$ and there is a central part around $p$  outside the compact set $B_{T\cup T',n}$, for times $(-\epsilon,\epsilon)$ . Now, following $\gamma$  along the flow-box with negative time starting at $\gamma(s_{I_k})$ we must arrive to $B_{T\cup T',n}$ but, since we are in the time interval $J_k$,  we reach exactly the point $\gamma(s_{J_k})$ that is in $B_{T,n}$.
 In the same way we can go in the positive time to reach $B_{T',n}$ exactly at $\gamma(t_{J_k})$. This shows that $\sigma^+_p$ intersects $B_{T,n}$ and $\sigma_p^-$ intersects $B_{T',n}$.
\section{Morse-Smale Line Foliations}
\label{sec:ms}
The classical two-dimensional conditions of Morse-Smale dynamical systems concern the singularities and the connections between them. The singularities are of hyperbolic type and saddle connections are avoided (for more details, see \cite{M-S}). In this section we define generic conditions for ${\mathcal L}, M, \partial M$  similar to
 Morse-Smale ones. Under that conditions, we shall see that $\omega(\gamma)$ is either a single point, a periodic orbit or the graph ${\mathcal C}(\gamma)$ is a single cycle.
\begin{definition}
\label{def:ms}
We say that ${\mathcal L}$ is of {\em weak Morse-Smale type} if and only if the following conditions are satisfied
\begin{enumerate}
\item The singularities of $\mathcal L$ form a finite set $S\subset \partial M$ and each one is hyperbolic, that is,   the eigenvalues have non-null real part.
\item There are no saddle connections in dimension two along $\partial M\setminus \partial ^2M$.
\item There are no infinitesimal saddle connections.
\end{enumerate}
\end{definition}
\begin{remark} The plane Morse-Smale conditions contain other requirements in addition to the hyperbolicity and the non-existence of saddle connections. More precisely, it is asked that the alpha and omega limit sets are either equilibrium points or periodic orbits; in our context we do not need to require this, in fact such conditions are automatically satisfied for the edges of $\omega(\gamma)$, in view of its graph structure.
In the papers \cite{Alo-C-C1, Alo-C-C2} the reader can find a study of the topological classification of three-dimensional real analytic vector fields whose reduction of singularities is of Morse-Smale type.
\end{remark}
Let us consider $\mathcal L$  of weak Morse-Smale type.  Along this section and the next one, we will explain the terms of Definition \ref{def:ms} and we will give a general local-global description of the situation.

Take a singular point $p\in S$, we have three possibilities
$$
p\in \partial M\setminus \partial^2M,\; p\in\partial^2M\setminus \partial^3M,\; p\in \partial^3M.
$$
Also $p$ may be an attractor (negative real part of the eigenvalues), a repelling point (positive real part of the eigenvalues) or a three-dimensional saddle (there are two eigenvalues with real parts of distinct sign).

Consider a singular point $p\in S\cup\omega(\gamma)$ and assume that $\omega(\gamma)$ has more than one point (in particular $\omega(\gamma)\not=\{p\}$). Then we have
\begin{enumerate}
\item The point $p$ cannot be an attractor or a repelling point of $\mathcal L$. More precisely, if $p$ is a repelling point of $\mathcal L$, then $p\notin\omega(\gamma)$ and if $p\in\omega(\gamma)$ is an attractor of $\mathcal L$, then $\omega(\gamma)=\{p\}$.
\item In the case that $p\in \partial M\setminus \partial^2M$, the restriction $\mathcal L\vert_{\partial M}$ of $\mathcal L$ to $\partial M$ cannot have an attractor or a repelling point at $p$.
    To see this, note that for $t>t_0$ the parameterized orbit $\gamma(t)$ is under a fixed ``height'' or distance to $\partial M$. Now, by the ``chimney shape'' of the dynamics around $p$ it is impossible to exit the ``chimney''
     once $\gamma(t)$, with $t>t_0$, is inside a neighborhood of $p$. That is, either $p\notin \omega(\gamma)$ or $\omega(\gamma)=\{p\}$.
\end{enumerate}
Thus, assuming that $\omega(\gamma)$ has more than one point, the vertices $p\in S\cap \omega(\gamma)$ of the graph ${\mathcal C}(\gamma)$ are of one of the following kinds:
\begin{enumerate}
\item {\em Saddle Corner.} The point $p$ belongs to $\partial ^3M$ and it is a three dimensional saddle. Note that the linear part of a generator of $\mathcal L$ at $p$ is diagonal if we write it in local coordinates $x,y,z$ such that $xyz=0$ is a local equation of $\partial M$, since the components of $\partial M$ are invariant for $\mathcal L$.
\item{\em Bi-Saddle Angle.}  The point $p$ belongs to $\partial ^2M\setminus \partial^3M$, it is a three dimensional saddle and the invariant variety of dimension one (stable or unstable) of $\mathcal L$ coincides locally at $p$ with the intersection of the two components $D_1$ and $D_2$ of $\partial M$ through $p$; that is, the invariant variety of dimension one coincides locally with the skeleton $\partial^2M$ at $p$. In particular the invariant variety of dimension two $\Delta$ is transversal to $\partial^2M$ and it gives two invariant local curves $\Gamma_i=\Delta\cap D_i$, $i=1,2$. More precisely, the restriction of $\mathcal L$ both to $D_1$ and $D_2$ is a two dimensional saddle.
\item{\em Saddle-Node Angle.} The point $p$ belongs to $\partial ^2M\setminus \partial^3M$, it is a three dimensional saddle and the invariant variety of dimension two (stable or unstable) of $\mathcal L$ at $p$ coincides with a component $D_1$ of $\partial M$ through $p$. The other component $D_2$ contains the invariant variety of dimension one. Thus, the restriction of $\mathcal L$ to $D_1$ is an attractor or a repelling point and the restriction to $D_2$ is a two dimensional saddle.
\item{\em Plane Saddle.}  The point $p$ belongs to $\partial M\setminus \partial^2M$, it is a three dimensional saddle and the invariant variety of dimension two (stable or unstable) of $\mathcal L$ at $p$ is transversal to $\partial M$ through $p$. The restriction of $\mathcal L$ to $\partial M$ has a two dimensional saddle at $p$.
\end{enumerate}
Now, we can precise condition (2) in Definition \ref{def:ms}: \begin{quote}

To say that there are no saddle connections in dimension two along $\partial M\setminus \partial ^2M$ means that given a component $D$ of $\partial M$ there are no bi-dimensional saddle connections of the restriction $\mathcal L\vert_D$ of $\mathcal L$ to $D$ along unstable-stable varieties contained in $D\setminus \partial^2 M$.
\end{quote}
\begin{remark}
We allow the existence of bi-dimensional saddle connections of  $\mathcal L\vert_D$ along the skeleton $\partial^2 M$. In order to understand why those connections are possibly ``generic", let us consider a typical situation of reduction of singularities, where $M$ has been obtained by repeated blow-ups
$
M\rightarrow {\mathbb R}^3.
$
In  the case of a single blow-up, it is possible to do a generic deformation of $\mathcal L$ in order to break saddle connections along the exceptional divisor, by using an method based on Melnikov's integrals. This is a blow-up version of the generic nature of Morse-Smale conditions. But when we have more than one blow-up, the intersection of two divisors is ``rigid'' for such type of deformations and we cannot break the saddle connections if that way. For more details, the reader can see \cite{Cam}.
\end{remark}

\begin{remark}
 \label{rk:verticesgrafo}
 Let $p\in {\mathcal C}(\gamma)$ be a hyperbolic singularity of $\mathcal L$ being a saddle corner, a bi-saddle angle, a saddle-node angle or a plane saddle. If there is an edge of ${\mathcal C}(\gamma)$ arriving to $p$, there is necessarily an edge exiting $p$. Moreover, the edges of ${\mathcal C}(\gamma)$ arriving to $p$ must be contained in the stable variety  and the edges exiting $p$ in the unstable variety. So the combinatorics of the possible options is the following one\begin{enumerate}
\item If $p$ is a saddle corner or a saddle-node angle and there are two or more edges arriving to $p$ (respectively exiting $p$), there is only one edge exiting $p$ (respectively arriving to $p$). This last one is an skeleton edge in the case of a saddle corner and a trace edge in the case of a saddle-node angle.
\item If $p$ is a bi-saddle angle or a plane-saddle, there is at most two edges arriving to $p$ and at most two edges exiting $p$. In the case of a bi-saddle angle, the edges exiting are of skeleton type if the edges arriving are of trace type and conversely. In the case of a plane saddle angle, the arriving and exiting edges are all of trace type.
\end{enumerate}
\end{remark}
We can graphically represent the four types of singularity in a plane way with arrows, a part of a circle for the two dimensional invariant variety and lines corresponding to the skeleton $\partial^2M$ (we recall that $\partial M$ is homeomorphic to ${\mathbb S}^2$).
\begin{center}
\strut\vspace{10pt}\par
\includegraphics[scale=0.5]{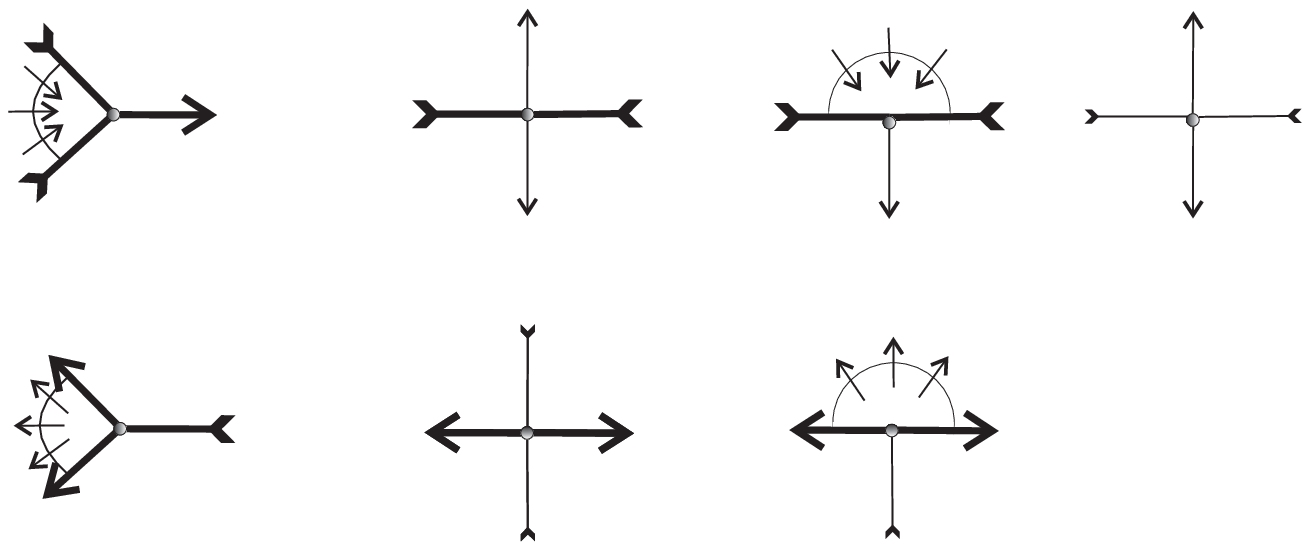}
\end{center}
\section{Weights and infinitesimal saddle connections}
\label{sec:weights}
In this section we give a quick overview of the definitions given in \cite{Alo-C-C1, Alo-C-C2} concerning infinitesimal saddle connections in dimension three.

Consider a parameterized real analytic curve
$$c:[1,\infty)\rightarrow {\mathbb R}^2_{>0}\subset {\mathbb R}^2_{\geq 0};\quad c(t)=(c_1(t),c_2(t)),$$ such that $c'(t)\ne0$ for all $t>1$ and $\lim_{t\rightarrow \infty }c(t)=(0,0)\in {\mathbb R}^2$. We say that $c$ has {\em weight $\rho>0$ at the origin} if there are constants $0<k_1<k_2<\infty$ and $t_0>1$ such that
$$
k_1 c_1(t)^\rho < c_2(t)< k_2 c_1(t)^\rho, \mbox{ for all } t>t_0.
$$
For a given subset $C\subset {\mathbb R}^2_{>0}$ we say that $C$ has weight $\rho$ if there is a parameterized curve $c$ of weight $\rho$ such that $C=\vert c\vert$.

Not all real analytic curves $c$ as above have a well defined positive weight. One can think about $c(t)=(t,\exp(-t))$, where the weight should be $\infty$ or an oscillating case as
$$c(t)=(t, \sin^2t^{\rho_1}+ \cos^2t^{\rho_2}),\quad 0<\rho_1<\rho_2,$$
where there is no weight at all. Note that if $c$ has weight $\rho$, then $\bar c$ has weight $1/\rho$ where $\bar c (t)=(c_2(t),c_1(t))$.
\begin{remark}\label{rmk:weightinvariance}
If $c$ has weight $\rho$, then $f\circ c$ has also weight $\rho$, for any germ of positively oriented analytic diffeomorphism $f:{\mathbb R}^2_{\geq 0}\rightarrow {\mathbb R}^2_{\geq 0}$.
 \end{remark}

Let us consider a vector field $\xi$ in a neighborhood of the origin of ${\mathbb R}^3$ given by
$$
\xi=(1+a(x,y,x))x\frac{\partial}{\partial x}+(-\lambda+b(x,y,x))y\frac{\partial}{\partial y}+ (-\mu+c(x,y,x))z\frac{\partial}{\partial z},
$$
where $a(0)=b(0)=c(0)=0$ and $\lambda,\mu\in {\mathbb R}_{>0}$. It gives a three-dimensional saddle at the origin, the unstable variety is the $x$-axis $y=z=0$, the stable variety is $x=0$ and the coordinate planes $xyz=0$ are invariant. We assume that $a,b,c$ are ``small enough'' to assure that $\xi$ is transversal to the planes $x=\epsilon$ and to the cylinders $y^2+z^2=\epsilon$, for $0<\epsilon\leq 1$, this is always possible up to a homothecy. Let us consider the inclusion
\begin{eqnarray*}
\phi_x:{\mathbb R}^2_{>0}\hookrightarrow {\mathbb R}^3;\quad \phi(u,v)=(1,u,v)
\end{eqnarray*}
and  the {\em fences}
\begin{eqnarray*}
\Phi:[0,1]\times[0,1]\rightarrow {\mathbb R}^3_{\geq 0}; \quad \Phi(u,v)=(v, \cos (\pi u/2),\sin (\pi u/2)).\\
\bar{\Phi}:[0,1]\times[0,1]\rightarrow {\mathbb R}^3_{\geq 0}; \quad \Phi(u,v)=(v,\sin (\pi u/2),\cos (\pi u/2)).
\end{eqnarray*}
Finally, let $B$ be the cylinder $0\leq x\leq 1, 0\leq y^2+z^2\leq 1, 0\leq y,z$.
The next proposition gives a dynamical interpretation of the definition of the weight
\begin{proposition}\label{prop:weight}{\cite{Alo-C-C1}} Consider a subset $C\subset {\mathbb R}^2_{>0}$ and let $S\subset {\mathbb R}^3_{>0}$ be the saturation  of $\phi_x(C)$ by the $\xi$-flow in $B$. The following statements are equivalent
\begin{enumerate}
\item $C$ has weight $\mu/\lambda$.
\item $\Phi^{-1}(S)$ is the support of a real analytic curve $$d:[1,\infty)\rightarrow [0,1]\times(0,1]$$ such that $\lim_{t\rightarrow \infty}d(t)=(\theta,0)$, where $0<\theta<1$.
\end{enumerate}
\end{proposition}
\begin{remark} The construction of proposition \ref{prop:weight} is reversible in the following sense. Consider a real analytic curve $$d:[1,\infty)\rightarrow [0,1]\times(0,1]$$ such that $\lim_{t\rightarrow \infty}d(t)=(\theta,0)$, where $0<\theta<1$ and let $D=\vert d\vert$ be its support. If $S$ is the saturation of $\Phi(D)$ by the $\xi$-flow in $B$, then $C=\phi_x^{-1}(S)$ has weight $\mu/\lambda$.
\end{remark}

We define $\mu/\lambda$ to be the {\em $(y,z)$-weight of $\xi$ at the origin,} once we chose the relative ordering of the invariant planes $y=0$, $z=0$; same definition for the reversed dynamics $-\xi$. Note that the  $(z,y)$-weight of $\xi$ at the origin is $\lambda/\mu$. In order to assign a well defined weight to each saddle corner, we fix once for all an ordering for the components of $\partial M$. In each saddle corner we choose local coordinates $x,y,z$ such that $\partial M$ is given by $xyz=0$, the coordinate plane $x=0$ is either stable or unstable and $y=0$ precedes $z=0$ in the fixed ordering. Note  that this global ordering allows us to define also a weight for the singularities of bi-saddle angle type in the same way.

We have the following {\em weight transitions} through three-dimensional saddles
\begin{proposition}\label{prop:invtransweight}{\cite{Alo-C-C1}} Consider a subset $C'\subset {\mathbb R}^2_{>0}$ of weight $\rho'\ne\mu/\lambda$ and let $S'$ be the saturation  of $\phi_x(C')$ by the $\xi$-flow in $B$. We have
\begin{enumerate}
\item If $\rho'>\mu/\lambda$, then $\Phi^{-1}(S')$ has weight $1/(\lambda\rho'-\mu)$.
\item If $\rho'<\mu/\lambda$, then $\bar\Phi^{-1}(S')$ has weight $\rho'/(\mu-\lambda\rho')$.
\end{enumerate}
Inversely, consider a subset $D'\subset {\mathbb R}^2_{>0}$ of weight $\rho''$ and let $S'_y$, respectively $S'_z$, be the saturation of $\Phi(D')$, respectively of $\bar\Phi(D')$, by the $\xi$-flow in $B$.
We have
\begin{enumerate}
\item[a)] $\phi_x^{-1}(S'_y)$ has weight $(1+\mu\rho'')/\lambda\rho''$.
\item[b)] $\phi_x^{-1}(S'_z)$ has weight $\mu\rho''/(1+\lambda\rho'')$.
\end{enumerate}
\end{proposition}
Let us go to the definition of infinitesimal saddle connections. A {\em positively orien\-ted weighted chain of length $n$} of saddles is a sequence
$$
{\mathcal S}=\{(p_{i-1},\sigma_i,\rho_i,p_i)\}_{i=1}^n
$$
satisfying the following properties:
\begin{enumerate}
\item Each $p_j$ is a saddle corner, for $j=1,2,\ldots, n-1$. Moreover $p_0$ and $p_n$ are either saddle corners or bi-saddle angles.
\item Each $\sigma_j$ is a parameterized orbit of $\mathcal L$ contained in $\partial^2M$, that gives an edge of $\mathcal L$ starting at  $p_{j-1}$ and ending at  $p_j$, for $j=1,2,\ldots,n$.
\item The weights $\rho_j>0$ and the edges $\sigma_j$ respect the transition rules stated in Proposition \ref{prop:invtransweight}, for $j=1,2,\ldots,n-1$.
\end{enumerate}
Let us explain the last statement. Consider a step $j$, with $1\leq j\leq n-1$. We take a transversal plane $\Delta_j$ to $\sigma_j$ and we draw a curve $C\subset \Delta_j$ of weight $\rho_j$ (recall that we have fixed an ordering for the two components of $\partial M$ containing $\sigma_j$). We do the saturation of $C$ through the singularity $p_j$. We assume that this saturation is adherent to $\sigma_{j+1}$ and that $\rho_{j+1}$ is the weight obtained by application of the corresponding case of Proposition  \ref{prop:invtransweight}. In particular there is one forbidden situation: it is not possible that $\rho_j$ is the weight of $p_j$ and simultaneously $\sigma_j$ is the invariant variety of dimension one for $p_j$, since in this case we apply Proposition \ref{prop:weight}.

The {\em reversed chain} $-{\mathcal S}$ of $\mathcal S$ is given by
$$
-{\mathcal S}=\{(q_{i-1},\tau_i,\omega_i,q_i)\}_{i=1}^n;\quad (q_{i-1},\tau_i,\omega_i,q_i)=(p_{n-i+1},\sigma_{n-i+1},\rho_{n-i+1},p_{n-i}).
$$
A {\em negatively oriented weighted chain} ${\mathcal T}$ is a chain of the form ${\mathcal T}=-{\mathcal S}$, where $\mathcal S$ is a positively oriented weighted chain.

We say that a weighted chain $\mathcal S$ (positively or negatively oriented) is {\em associated to $p_0$} if and only if $\rho_1$ is the weight of $p_0$ and $\sigma_1$ is the invariant variety of dimension one for $p_0$. Two such sequences of lengths $n\leq n'$ coincide up to the step $n$.
\begin{definition} An {\em infinitesimal saddle connection} is a weighted chain of saddles
$$
{\mathcal S}=\{(p_{i-1},\sigma_i,\rho_i,p_i)\}_{i=1}^n
$$
that is associated to $p_0$ and whose reversed chain $-{\mathcal S}$ is also associated to $p_n$.
\end{definition}
In this way we have a precise meaning for all the conditions in Definition \ref{def:ms}.
\section{Accumulation under weak Morse-Smale conditions}
In order to prove Theorem \ref{th:msth}, we assume that $\omega(\gamma)$ is neither a point nor a periodic orbit.  We have the following  possibilities
\begin{enumerate}
\item All edges in ${\mathcal C}(\gamma)$ are trace edges.
\item There are skeleton and trace edges in ${\mathcal C}(\gamma)$.
\item All edges in ${\mathcal C}(\gamma)$ are skeleton edges.
\end{enumerate}
In this section we consider the case that all  edges in ${\mathcal C}(\gamma)$ are trace edges, the other situations are considered in next section.

Let us start with an edge $\sigma$. Then  $\sigma$ is contained in a single component $D$ of $\partial M$. Let $p_0,p_1$ be respectively the origin and the end of $\sigma$. Since there are no bi-dimensional saddle connections along $\partial M\setminus \partial^2M$, we have one of the following situations (not mutually exclusive)
\begin{enumerate}
\item[i)] ${\mathcal L}\vert_D$ has a repelling point at $p_0$.
\item[ii)] ${\mathcal L}\vert_D$ has an attractor at $p_1$.
\end{enumerate}
Let us consider first the case ii), where $p_1$ is an attractor of ${\mathcal L}\vert_D$. The only possibility is that $p_1$ is a saddle-node angle, where $D$ is the invariant variety of dimension $2$ and moreover the only edge exiting $p_1$ is the invariant variety $\sigma_2$ of dimension one, contained in another component $D_2$ of $\partial M$. To see this, note that if $p_1$ were a saddle corner, then the only exiting edge $\sigma_2$ should be contained in the skeleton $\partial^2M$. Now, the origin $p_1$ of $\sigma_2$ is a bi-dimensional saddle for ${\mathcal L}\vert_{D_2}$, this implies that the end $p_2$ of $\sigma_2$ is an attractor of ${\mathcal L}\vert_{D_2}$. We repeat the argument at $p_2$. In this way the obtain an infinite sequence
$$
\sigma=\sigma_1,p_1,\sigma_2,p_2,\sigma_3,p_3,\ldots
$$
The number of equilibrium points is finite. Thus, not all the points $p_i$ are distinct and, up to renumber points and edges, we have a finite sequence
$$
{\mathcal C}=\{
p_0,\sigma_1,p_1,\sigma_2,p_2,\sigma_3,p_3,\ldots,p_{n-1},\sigma_n,p_n=p_0\}.
$$
where all the points are distinct and $n\geq 1$. Hence, for each $i=1,2,\ldots,n$, we have that $\sigma_i$ is an edge of ${\mathcal C}(\gamma)$ contained in a single component $D_i$ of $\partial M$, the origin $p_{i-1}$ of $\sigma_i$  is a bi-dimensional saddle of ${\mathcal L}\vert_{D_i}$ and the end $p_i$ of $\sigma_i$ is a bi-dimensional attractor of ${\mathcal L}\vert_{D_i}$.

\begin{center}
\strut\vspace{10pt}\par
\includegraphics[scale=0.4]{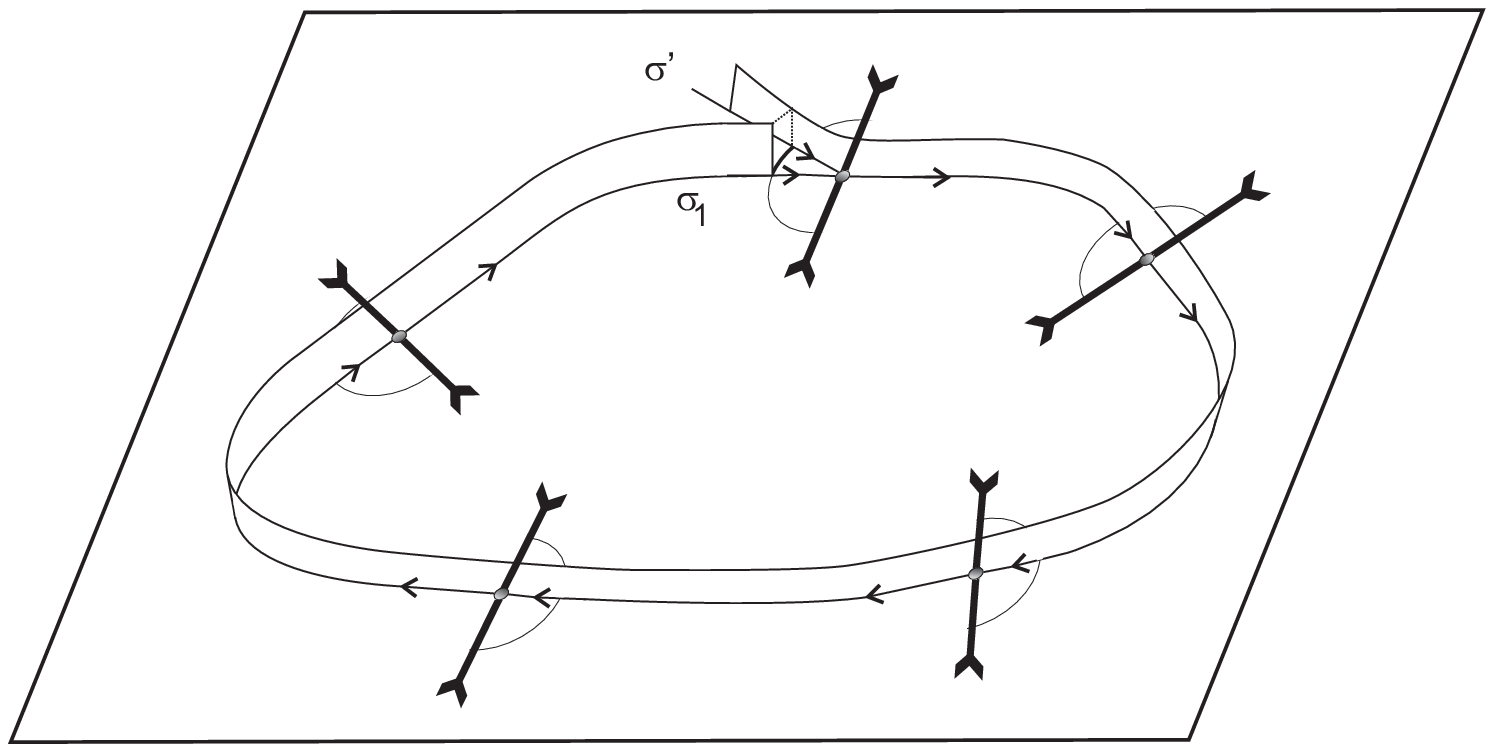}
\par\strut\vspace{10pt}\par
\end{center}

Now, let us see that the sequence $\mathcal C$ is the whole graph ${\mathcal C}(\gamma)$, and thus we have a single cycle. Let us reason by contradiction assuming that ${\mathcal C} \ne{\mathcal C}(\gamma)$. Since ${\mathcal C}(\gamma)$ is connected, there is an edge $\sigma'$ different from $\sigma_1,\sigma_2,\ldots,\sigma_n$ such that either the origin or the end of $\sigma'$ is one of the points $p_0,p_1,\ldots,p_{n-1}$. But the only edge exiting $p_i$ is $\sigma_i$ and thus, up to a new circular renumbering, we can assume that $p_1$ is the end point of $\sigma'$ and $\sigma'\ne\sigma_1$. Consider a small semi-disc $\Delta\subset D_1$ around $p_1$. Then $\sigma_1$ and $\sigma'$ cut the frontier ${\mathbb S}$ of $\Delta$ in two different points $q_1,q'$, where $q_1=\sigma_1(t_0)$. Consider a small analytic curve $C$ transversal to $D_1$ at $q'$ and let $\sigma'_+$ be the positive parameterized orbit starting at $q'$, whose end point is $p_1$. Let $S$ be the saturation of $C$ by the flow. Near $p_1$, we see that $\sigma_2$ is contained in the adherence of $S$ and in fact coincides locally with the intersection of the adherence of $S$ with $D_2$, we go with the flow near $p_2$ and we repeat the argument. In that way, we obtain a Poincar\'{e}-Bendixson trap, whose base is $$\vert\sigma'_+\vert\cup\vert\sigma_2\vert\cup \vert\sigma_3\vert\cdots \vert\sigma_n\vert\cup \sigma_1((-\infty,t_0])\cup \varsigma,$$
where $\varsigma$ is the segment of ${\mathbb S}$ joining $q_1$ and $q'$.
The fact that $\sigma_1([t_0,\infty))$ is in $\omega(\gamma)$ excludes that $\sigma'((-\infty,0))$ is in $\omega(\gamma)$, in view of Lemma \ref{lema:uno}. This is the desired contradiction.

Assume now that ii) never holds. In this situation the above argument works just by reversing the ordering in ${\mathcal C}(\gamma)$.
\section{Accumulation along the skeleton}
In this section we finish the proof of Theorem \ref{th:msth} by considering the situations where we have skeleton edges of ${\mathcal C}(\gamma)$.

{\em Second case: There are skeleton edges and trace edges in ${\mathcal C}(\gamma)$.} In view of the dynamics described in Remark \ref{rk:verticesgrafo}, we have that either there is a trace edge $\sigma$ of ${\mathcal C}(\gamma)$ ending at a point $p$ and a skeleton edge $\sigma'$ starting at $p$, or a skeleton edge $\tau'$ of ${\mathcal C}(\gamma)$ ending at a point $p$ and a trace edge $\tau$ starting at $p$. Moreover, by using arguments as in the proof of Proposition \ref{prop:ciclico}, both properties are satisfied. Now, we have two possibilities
\begin{enumerate}
\item[A)] There is a trace edge $\sigma$ of ${\mathcal C}(\gamma)$ arriving to a saddle corner.
\item[B)] No trace edge of  ${\mathcal C}(\gamma)$  arrives to a saddle corner. In particular, the trace edges arrive to saddle-node angles, bi-saddle angles or plane saddles.
\end{enumerate}
Let us consider the case A). We consider a trace edge $\sigma$ arriving to a saddle corner $p$. Then $\sigma$ is contained in a component $D$ of $\partial M$. Consider a small quadrant of disc $\Delta\subset D$ around $p$ and let $q$ be the point of intersection of $\sigma$ with the frontier of $\Delta$. Take two points $r_1$ and $r_2$ on the frontier of $\Delta$ near $q$ each one at one side of $q$, and consider two small analytic curves $C_1$, $C_2$ transversal to $D$ at $q_1$ and $q_2$ respectively.  Now, we are going to consider the saturation  $S_1$ and $S_2$ of $C_1$ and $C_2$ respectively by the flow of ${\mathcal L}$, going step by step along the singular points of $\mathcal L$. Let $\sigma_1$ be the edge corresponding to the invariant variety of dimension one of $p$. We know that $\sigma_1$ is an edge of ${\mathcal C}(\gamma)$ and if we cut $S_1,S_2$ by a small transversal to $\sigma_1$, we obtain two curves $C^1_1,C^1_2$ of weight $\rho_1$, where $\rho_1$ is the weight associated to the point $p$; in between, the trajectory $\gamma$ passes infinitely many times. Now, let us go to the end point $p_1$ of $\sigma_1$. Assume first that $p_1$ is a saddle corner. The transition of weights described in Proposition \ref{prop:invtransweight} give to us a new edge $\sigma_2$ where the saturations $S_1$ and $S_2$ are adherent. This is because there are no infinitesimal saddle connections. We obtain two new curves $C_1^2,C_2^2$ of weight $\rho_2$ on a transversal plane to $\sigma_2$. Moreover, $\sigma_2$ is an edge of ${\mathcal C}(\gamma)$, since $\gamma$ passes infinitely many times between $C_1^2$ and $C_2^2$. Let $p_2$ be the end point of $\sigma_2$. If $p_2$ is a saddle corner, we repeat the argument. Now, we have two possibilities: either we encounter only saddle corners in this way or not. Let us go to the second case. After finitely many steps, the end point $p_k$ of $\sigma_k$ is not a saddle corner. It must be of angle type. Anyway, the saturations $S_1$ and $S_2$ go over a trace edge $\sigma_{k+1}$ contained in a component $D'$ of $\partial M$ and such that $\sigma_{k+1}$ is the unstable  variety of the bi-dimensional saddle ${\mathcal L}\vert_{D'}$. We get two curves $C^{k+1}_1,C^{k+1}_2$ over a small transversal to $\sigma_{k+1}$ such that $\gamma$ passes in between infinitely many times. Now, we go to the next point $p_{k+1}$. Since there are no bi-dimensional saddle connections, the trace edge $\sigma_{k+1}$ is in the two dimensional stable variety of $p_{k+1}$. If $p_{k+1}$ is a saddle corner, we re-start the argument as at the initial point $p$. If $p_{k+1}$ is a saddle-node angle, the edge $\sigma_{k+1}$ in in the side {\em node} and the exiting edge $\sigma_{k+2}$ in the side {\em saddle}. The situation repeats.

Hence, the above construction give to us a step by step construction of $S_1,S_2$ that coincides along $\sigma_1,\sigma_2,\ldots$ and such that $\gamma$ passes in between infinitely many times in a transversal section to each of the $\sigma_k$, $k=1,2,\ldots$. Since we have only finitely many singular points, there is a first repetition and we get a cycle. This allows us to construct two Poincar\'{e}-Bendixson traps, one from $S_1$ and the other one from $S_2$, with base $B$ the support of this cycle. Moreover, since $\gamma$ passes in between $S_1$ and $S_2$, the accumulation set $\omega(\gamma)$ is outside the two components of the complement of $B$, hence it coincides with $B$.

Let us consider now the case B). There is a trace edge $\sigma$ arriving or exiting a saddle-node angle in the side {\em node}, otherwise, given a trace edge $\sigma$, the starting point and the end point of $\sigma$ are both bi-dimensional saddles, but we have no connections of bi-dimensional saddles along $\partial M\setminus \partial^2 M$. Assume thus that $\sigma$ arrives to a saddle-node angle $p$ in the side {\em node}. The argument for the situation that $\sigma$ is exiting a saddle-node angle $p$ in the side {\em node} works in the same way. There is a unique $\sigma_1$ exiting $p$ that must coincide with the invariant variety of dimension one of $p$, then $\sigma_1$ is an edge of ${\mathcal C}(\gamma)$. The end point of $\sigma_1$ must be a saddle-node angle $p_2$ in the side ``node''. We continue the argument at $p_2$. In this way, up to reordering the indices, we obtain a cycle of edges and points
$$
p_0,\sigma_1,p_1,\sigma_2,p_2,\ldots,p_{n-1},\sigma_n,p_n=p_0,
$$
where $\sigma_k$ is a trace edge of ${\mathcal C}(\gamma)$, with initial point $p_{k-1}$ and end point $p_k$ and each $p_k$ is a saddle-node angle where $\sigma_k$ is on the side ``node'' and $\sigma_{k+1}$ is on the side ``saddle''. Now, we can prove that this cycle is necessarily the whole graph ${\mathcal C}(\gamma)$ with the same arguments as in the first case where there are no skeleton edges. Note that ``a posteriori" we have no skeleton edges and thus this possibility B) never holds.

{\em Third case: all edges in ${\mathcal C}(\gamma)$ are skeleton edges.} Hence all the vertices of ${\mathcal C}(\gamma)$ are saddle corners. First we will construct a special simple cycle ${\mathcal C}$ in ${\mathcal C}(\gamma)$ and second we will prove that $\mathcal C$ coincides with all the graph ${\mathcal C}(\gamma)$. The cycle ${\mathcal C}$ should be of the form
$$
{\mathcal C}:
 p_0,\sigma_1,p_1,\sigma_2,p_2,\ldots,p_{n-1},\sigma_n,p_n=p_0
$$
where $p_{k}$ is the initial point of the edge $\sigma_{k+1}$ and $p_{k+1}$ its end point, for each $k$ with $0\leq k\leq n-1$. Moreover $p_i\ne p_j$ if $i\ne j$ and $0\leq i,j\leq n-1$. In this case, the support $\vert {\mathcal C}\vert $ is homeomorphic to ${\mathbb S}^1$ and it is naturally oriented by ${\mathcal L}$. Once we have fixed an orientation of $\partial M$ we can speak about the connected component $U$ of $\partial M\setminus \vert {\mathcal C}\vert $ corresponding to the {\em right hand side} of the oriented support $\vert {\mathcal C}\vert $. Then, we require the following minimality property to $\mathcal C$:
\begin{quote} {\bf M}:{\em \/There is no edge of ${\mathcal C}(\gamma)$ contained in $U$.}
\end{quote}
Let us show how to construct ${\mathcal C}$ with the required property {\bf M}. Take an edge $\sigma_1$ ending at a point $p_1$. Recall that $p_1$ is a corner point, so we have two possibilities:\begin{enumerate}
\item The edge $\sigma_1$ is contained in the invariant variety of dimension two of $p_1$. In this case there is exactly one edge $\sigma_2$ exiting the point $p_1$.
\item The edge $\sigma_1$ coincides with the invariant variety of dimension one of $p_1$. Now, we have also two possibilities
\begin{enumerate} \item There is only one edge $\sigma_2$ exiting $p_1$.
\item There are exactly two edges exiting $p_1$. In this case we chose one of them, that we call $\sigma_2$, by using the {\em right hand side criterion}, that is between the two possibilities, we chose the one ``on the right'' of $\sigma_1$ for a once for all fixed orientation in ${\partial M}$, recall that ${\partial M}$ is homeomorphic to ${\mathbb S}^2$.
\end{enumerate}
\end{enumerate}
\begin{center}
\strut\vspace{10pt}\par
\includegraphics[scale=0.4]{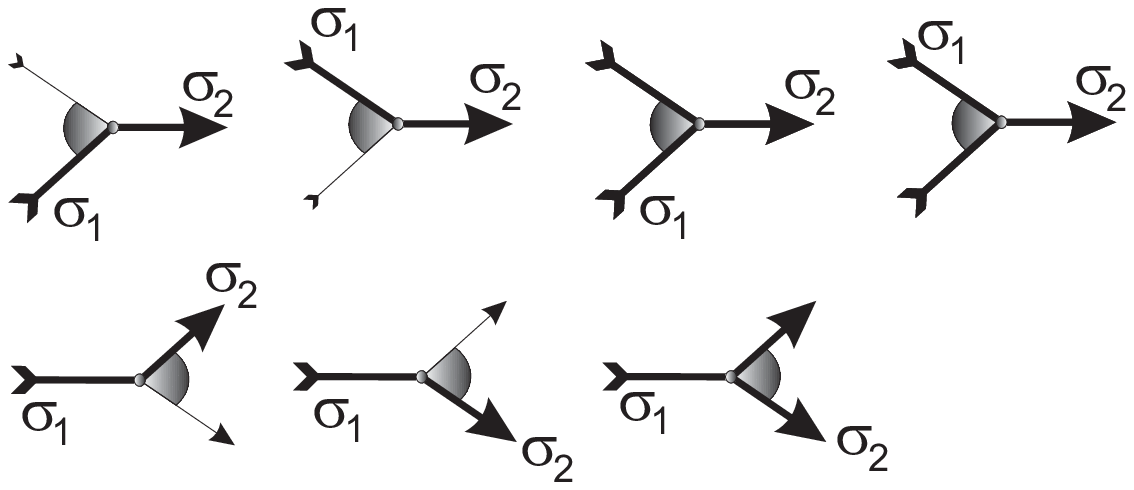}\\
{\small The seven possibilities for $\sigma_1$, $\sigma_2$}
\par\strut\vspace{10pt}\par
\end{center}
Now $\sigma_2$ has ending point at $p_2$. We continue in this way and we stop the first time we repeat a vertex. Up to reordering the indices, we obtain a simple cycle
$$
 {\mathcal C}_1: p_0,\sigma_1,p_1,\sigma_2,p_2,\ldots,p_{n-1},\sigma_n,p_n=p_0
$$
Let $U_1$ be the right hand side component of $\partial M\setminus \vert{\mathcal C}_1\vert$. If $U_1\cap \vert{\mathcal C}(\gamma)\vert=\emptyset$, we are done. Otherwise, by connectedness of $\omega(\gamma)$, there is a point $p_k$ with  $0\leq k\leq n-1$ such that there is an edge $\sigma'$ exiting or arriving to $p_k$ and contained in $U_1$. It is not possible that $\sigma'$ has $p_k$ as initial point, since by the {\em right hand side} criterion, we should have that $\sigma'=\sigma_{k+1}$ and thus $\sigma'$ is contained in $\vert {\mathcal C}\vert$. So, the only possibility is that $p_k$ is the final point of $\sigma'$. Call $q_{-1}$ the initial point of $\sigma'$. Now, $q_{-1}$ cannot coincide with one of the points $p_i$, because of our criterion of the right hand side, since otherwise the criterion is not respected at $p_i$. Take $\sigma'_{-1}$ arriving to $q_{-1}$ and call $q_{-2}$ the initial point of $\sigma'_{-1}$, by the same argument as before the point $q_{-2}$ does not coincide with one of the points $p_i$. We continue in this way and up to a moment, we produce a simple cycle  contained in $U_1$
$$
{\mathcal D}_1: q_{-k},\sigma'_{-k+1},q_{-k+1},\sigma'_{-k+2},q_{-k+2},\ldots, q_{-k'-1},\sigma'_{-k'},q_{-k'}=q_{-k}.
$$
Note that the right hand side connected component $V_1$ of $\partial M\setminus \vert{\mathcal D}_1\vert$ is strictly contained in $U_1$, in fact, we have that $q_{-k}\in U_1\setminus V_1$. Now we re-start with $\sigma'_{-k+1}$ to produce a cycle ${\mathcal C}_2$ that respects the right hand side criterion. Let $U_2$ be the right hand side connected component of $\partial M\setminus \vert{\mathcal C}_2\vert$. We have that
$$
U_1 \varsupsetneq V_1\supset U_2.
$$
By continuing in this way we get a cycle ${\mathcal C}={\mathcal C}_r$ satisfying the property ({\bf M}).

\begin{center}
\strut\vspace{10pt}\par
\includegraphics[scale=0.4]{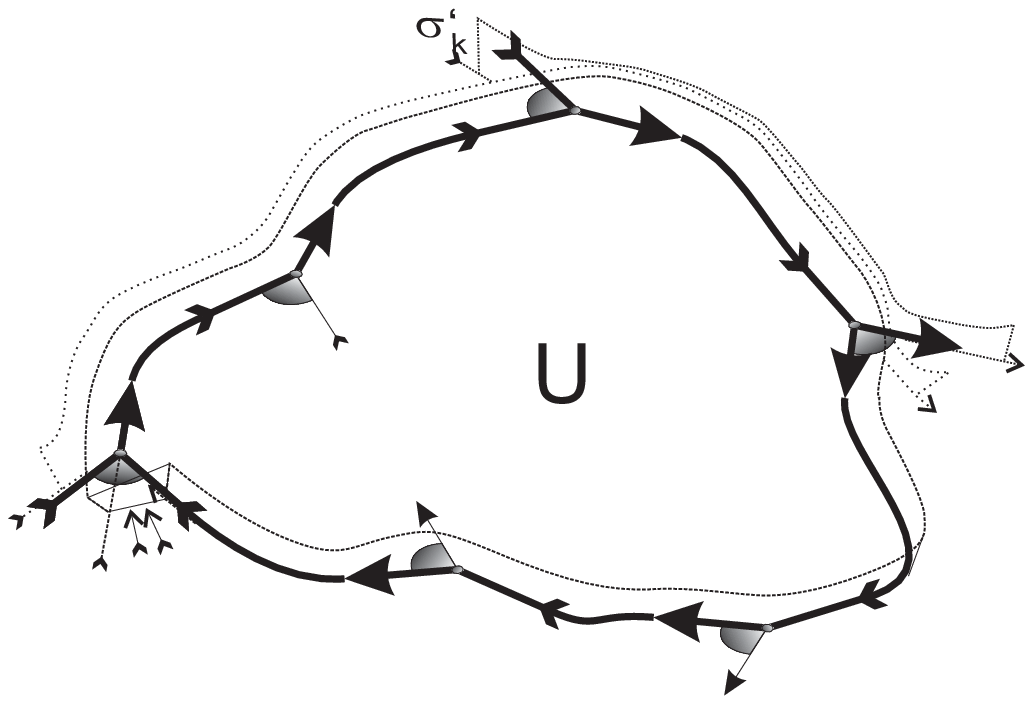}
\par\strut\vspace{10pt}\par
\end{center}

Now, we are going to see that ${\mathcal C}$ is the whole graph ${\mathcal C}(\gamma)$. Denote by $\widetilde U$, $U$ the two connected components of $\partial M\setminus \vert{\mathcal C}\vert$; note that $U$ does not contain any edge of ${\mathcal C}(\gamma)$.  Assume that ${\mathcal C}$ is not the whole graph ${\mathcal C}(\gamma)$. By connectedness, we have two possibilities
\begin{itemize}
\item There is an index $k$ and an edge $\sigma'_k$ in ${\mathcal C}(\gamma)$ having $p_k$ as end point with $\sigma'_k\ne\sigma_k$. Necessarily $\sigma'_k$ is contained in $\widetilde U$.
    \item There is an index $k$ and an edge $\sigma'_k$ in ${\mathcal C}(\gamma)$ having $p_{k-1}$ as starting point with $\sigma'_k\ne\sigma_k$. (This corresponds to the same situation as before, but reversing the arrows). We also have that  $\sigma'_k$ is contained in $\widetilde U$.
\end{itemize}
Assume that we are in the first case, the second case can be handled in a similar way.
Then,  $\sigma_k$ is inside the invariant variety of dimension two of $p_k$. If we take a non skeleton point $q_k$ in the invariant variety of dimension two of $p_k$, then $q_k$ is on the left hand side of $\sigma_k$ and on the right hand side of $\sigma'_k$. Consider the positive saturation $S_k$ of a curve $\Gamma_k$ transversal to $\partial M$ at the point $q_k$. We see that $S_k$ contains $\sigma_{k+1}$ in the adherence and $\gamma$ goes ``on the right'' of $S_k$. Once we arrive to $p_{k+1}$, we have two possible situations: either $S_k$ follows $\sigma_{k+2}$ or $S_k$ follows another edge $\sigma'_{k+2}$ on the left hand side of of $\sigma_{k+2}$. In the first case, we continue our argument going to $p_{k+2}$. Note anyway that $S_k$ always follows the weight of $p_k$ and because there are no infinitesimal saddle connections, the saturation $S_k$ accumulates only at the skeleton, following the {\em path of $p_k$} (for more details see \cite{Alo-C-C1, Alo-C-C2}). If we close the cycle in this way, we see that $\gamma$ is always on the right hand side of $S_k$, this avoid the existence of other edges in ${\mathcal C}(\gamma)$. Assume that we encounter a point $p_{k'}$ with $k'>k+1$ such that $S_k$ follows another edge $\sigma'_{k'}$ on the left hand side of $\sigma_{k'}$. Now we take a small transversal curve $\Gamma_{k'}$ in a non skeleton point $q_{k'}$ of the invariant variety of dimension two of $p_{k'}$ and we do the inverse saturation $S'_{k'}$ of $\Gamma_{k'}$. Then $S'_{k'}$ contains $\sigma_{k'-1}$ in the adherence and follows the weight of $p_{k'}$, going downstairs, we encounter a possible bifurcation between the path of $p_{k'}$ and the cycle $\mathcal C$ at a point $p_{k''}$ with $k''<k$. We re-start in a direct way from $p_{k''}$. At one of these steps, we close the cycle. In this way we obtain a Poincar\'{e}-Bendixson trap such that $\omega(\gamma)$ is contained in the union of the ``right hand side'' connected component $U$ and the cycle, but there is no edges of ${\mathcal C}(\gamma)$ inside $U$. Thus ${\mathcal C}(\gamma)$ coincides with $\mathcal C$.
\begin{remark} If is reasonable to expect the same kind of results if we substitute the condition that there is no infinitesimal saddle connections  by the the fact that all the singularities are linearizable. In this case, the saturation of a curve should transversely live in an o-minimal structure and there is no oscillation under a return, in the spirit of Khovanskii theory \cite{Kho} (see also \cite{Rol-S-S}).
\end{remark}

\end{document}